\documentclass[smallextended]{svjour3}       % onecolumn (second format)
\usepackage{amsfonts,amssymb,amscd,amsmath,enumerate,verbatim,calc}

\let\~\tilde

\newtheorem{thm}{Theorem}[section]
\newtheorem{cor}[thm]{Corollary}

\begin{document}

\title{A note on 2-dimensional Finsler manifold}
\author{Morteza Faghfouri \and Rahim Hosseinoghli}
\institute{M. Faghfouri \at
              Faculty of Mathematics, University of Tabriz, Tabriz, Iran.\\
              \email{faghfouri@tabrizu.ac.ir}           %  \\
           \and
           R. Hosseinoghli \at
Faculty of mathematics, University of Tabriz, Tabriz, Iran.\\
\email{r\_ hosseinoghli90@ms.tabrizu.ac.ir}}
\date{}

\maketitle
\begin{abstract}
We solve the following problem for $n=2:$ Is any n-dimensional Finsler manifold $(M, F)$ with a function $f$ which is nonconstant and smooth on $M$ satisfying
$
\dfrac{\partial g^{ij}}{\partial y^k}\dfrac{\partial f}{\partial x^i}=0,
$
a Riemannian  manifold?
The problem for $n>2$ remains open.
\end{abstract}

\section{Introduction}
The notion of doubly warped product manifolds has an important role in Riemannian geometry
and its applications. For example, Beem-Powell in \cite{Beem-Powell} studied this product for Lorentzian manifolds.
Then Allison in \cite{Allison:Pseudoconvexity} considered global hyperbolicity of doubly warped products and null pseudo convexity
of Lorentzian doubly warped products and recent years in \cite{unal:doubly.warped.products}, %\cite{olteanu:generalinequalityfordoublywarped},
\cite{faghfouri2013doubly} and \cite{faghfouri:ondoublyimmersion} extended  some properties of warped product, submanifolds  and geometric inequality in warped product manifolds  for doubly warped product submanifolds into arbitrary Riemannian manifolds. In 2001, Kozma- Peter-Varga in \cite{kozma:WarpedproductofFinslermanifolds} defined their warped product
for Finsler metrics and concluded that completeness of a doubly warped product can be related to
completeness of its components.

In \cite[Theorem 3]{PeyghanTayebi:OndoublywarpedproductFinslerManifold} %\cite{PeyghanTayebi:OndoublywarpedproductFinslerManifold}
 E. Peyghan and A. Tayebi
 % established a sharp relationship between the warping functions $f_1$ and $f_2$ of a doubly warped product Finsler manifold ${M_1}_{f_2}\times _{f_1} M_2$
 %and B. Najafi
  proved that: Let ${M_1}_{f_2}\times _{f_1} M_2$ be  a  $DWP $-Finsler manifold and $f_1$ is constant on $M_1$ ($f_2$ is constant on $M_2$). Then ${M_1}_{f_2}\times _{f_1} M_2$ is Berwaldian if and only if $M_1$ is Riemannain, $M_2$ is Berwaldian and
 \begin{align}
C^{ij}_k\dfrac{\partial f_1}{\partial x^i}=0.\tag{$*$}\label{$*$}
\end{align}
($M_2$ is Riemannian, $M_1$ is Berwaldian and $C^{ij}_k\dfrac{\partial f_2}{\partial x^i}=0$).\\
Now we can ask this question: Is there any nonconstant smooth function on Finsler manifold $ M $, which satisfies \eqref{$*$}?

In this paper by using partial differential equation properties,  we show that, if $ M $ is a 2-dimensional Finsler manifold and the equality \eqref{$*$}  holds for a nonconstant function $ f $, then
$ M $ is a Riemannian manifold.

%First we introduce some notions and   preliminaries.
%%%%%%%%%%%%%%%%%%%%%%%%%%%%%%%%%%%%%%%%%%%5
\section{Preliminaries}
Let $M$ be a $n$-dimensional $C^\infty$ manifold. Denote by $T_xM$ the tangent space
at $x \in M$, by $ TM:=\bigcup _{x\in M}T_{x}M$ the tangent bundle of $ M$, and by $TM^0 =
TM - \lbrace 0\rbrace$ the slit tangent bundle on $ M$. A Finsler metric on $M $ is a function
$ F:TM\to  [0,\infty) $ which has the following properties:
\begin{enumerate}
\item[(i)] $F$ is $C^\infty$ on $ TM_0 $;
\item[(ii)] $F$ is positively 1-homogeneous on the fibers of tangent bundle $TM$;
\item[(iii)] for each $ y \in T_xM $ , the following quadratic form $g_y$ on $T_xM$ is positive
definite,
\end{enumerate}
where
$$g_{y}(u,v)=\dfrac{1}{2}\dfrac{\partial ^2}{\partial t\partial s}[F^{2}(y+su+tv)]\vert_{s,t=0}.$$
%Let $x \in M$ and $F_x := F|_{T_xM}$.\\
Let $(M, F)$ be a Finsler manifold. The second and third order derivatives
of $ \frac{1}{2}F^2_x :=\frac{1}{2}F^2(x,y)$
at $ y \in T_xM^0 $ are the symmetric  forms $ g_y $ and $ C_y $  on $ T_xM $ ,
which called the fundamental tensor and Cartan torsion, respectively. in other notation,
$$ C_{y}: T_{x}M\times T_{x}M\times T_{x}M \to \mathbb{R} $$
$$C_{y}(u,v,w):=\dfrac{1}{2}\dfrac{d}{dt}\left[ g_{y+tw}(u,v)\right] \vert_{t=0}, \qquad u,v,w\in T_{x}M $$
the family $ C:=\lbrace C_{y}\rbrace _{y \in TM_{0}} $
 is called the cartan torsion, it is well known that $ C=0$ if and only if $ F $ is Riemannian.
 let $ b_i $ be a local frame for $TM$, and $g_{ij}:=g_y(b_i,b_j) $, $ C_{ijk} := C_y(b_i,b_j,b_k)$.
 then $ g_{ij}(x,y) = \dfrac{1}{2}\dfrac{\partial ^{2}F^{2}(x,y)}{\partial y^{i}\partial y^{j}} $ and $
 C_{ijk}=\frac{1}{2}\frac{\partial g_{ij}}{\partial y^k}=\frac{1}{4}\frac{\partial^3 F^2}{\partial y^i\partial y^j\partial y^k} $.
%%%%%%%%%%%%%%%%%%%%%%%%%%%%%%%%%%%%%%%%%%%%%
For a Finsler manifold $(M, F)$, a global vector filed $ G $ is induced by $ F $ on $ TM^0 $, which in a standard coordinate $ (x^i,y^i) $
 for $ TM^0 $ is given by $  G=y^i\dfrac{\partial}{\partial x^i}-2C^i(x,y)\dfrac{\partial}{\partial y^i} $,
 where
  % differential equations $ \ddot{c}^{i}+2\dot{G}^{i}(\dot{c})=0$, where the local functions $ G^{i}=G^{i}(x,y) $ are
%called the spray coeffcients and given by following
%Let $ F $  be a Finsler metric on an n-dimensional manifold $ M $ and $ G^i $ the geodesic coefficients of $ F $, wich are defined by
$$G^i=\dfrac{1}{4}g^{il}\lbrace [F^2]_{x^ky^l}y^k-[F^2]_{x^l}\rbrace, \qquad y\in T_xM.$$
The $ G $ is called the spry associated to $ (M,F) $.
A Finsler metric $F$ is called a Berwald metric if $ G^i=\frac{1}{2} \Gamma^{i}_{jk}(x)y^jy^k $  is quadratic in
$y \in T_xM$ for any $x \in M$.
For a tangent vector $y \in T_xM^0$, define $B_{y}:T_{x}M\times T_{x}M\times T_{x}M \to T_{x}M $, $E_{y}:T_{x}M\times T_{x}M\to \mathbb{R} $ and $D_{y}:T_{x}M\times T_{x}M\times T_{x}M \to T_{x}M $  by
$$ B_{y}(u,v,w) := B^{i}_{jkl}(y)u^{j}v^{k}w^{l}\dfrac{\partial}{\partial x^{i}}\vert_x,  E_{y}(u,v) := E_{jk}(y)u^{j}v^{k}$$ and $$ D_{y}(u,v,w) := D^{i}_{jkl}(y)u^{i}v^{k}w^{l}\dfrac{\partial}{\partial x^{i}}\vert_x$$
where
$$ B^{i}_{jkl} :=\dfrac{\partial ^{3}G^{i}}{\partial y^{j}\partial y^{k}\partial y^{l}},\quad E_{jk}=\frac{1}{2}B^m_{jkm}, $$
$$ D^i_{jkl} := B^i_{jkl}-\frac{2}{n + 1}\{E_{jk}\delta^i_l + E_{jl}\delta^i_k + E_{kl}\delta^i_j+\frac{\partial E_{jk}}{\partial y^l} y^i\}.$$
 $B$, $E$ and $D$ are called the Berwald curvature,  mean Berwald curvature and   Douglas curvature, respectively.  Then $ F$ is called a Berwald metric,  weakly Berwald metric  and a Douglas metric if
$B = 0$, $E = 0$ and $D=0$, respectively\cite{chern.shen.Ri.Fin}.
%%%%%%%%%%%%%%%%%%%%%
The notion of warped product manifold was introduced in \cite{bishop.oneill:} where it
served to give new examples of Riemannian manifolds.
On the other hand, Finsler geometry is just
Riemannian geometry without the quadratic restriction. Thus it is natural to
extending the construction of warped product manifolds for Finsler geometry\cite{kozma:WarpedproductofFinslermanifolds}.

Let $ (M_1, F_1) $  and $ (M_2, F_2) $ be two Finsler manifolds and $ f_i : M_i \to\mathbb{R}^+, i=1,2$
 are  smooth functions. Let  $ \pi _i : M_1 \times M_2 \to  M_i, i=1,2 $ % and $ \pi _2 : M_1 × M_2 \to  M_2 $
 be the natural projection maps. The product manifold
$ M_1 \times M_2 $ endowed with the metric $ F : TM^0_1 \times TM^0_2 \to  \mathbb{R}  $  given by
 $$ F (y , v ) =\sqrt{  f^{2}_2(\pi_{2}(y)) F_{1}^{2}(y) + f^{2}_1(\pi_{1}(y))F_{2}^{2}(v)}  $$
is considered, where $ TM^0_1 = TM_1- \lbrace 0\rbrace $ and $ TM^0_2 = TM_2-\lbrace 0 \rbrace $ . The metric defined above is a
Finsler metric. The product manifold $ M_1 \times M_2 $ with the metric $ F(\textbf{y}) = F(y, v) $
for $ (y, v) \in TM^0_1 \times TM^0_2 $  defined above will be called the doubly  warped
product (DWP) of the manifolds $ M_1 $  and $ M_2 $  and  $ f_i, i=1,2 $  will be called the warping
function. We denote this  warped by $ {M_1}_{f_2}\times _{f_1} M_2 $. % If  $ f_1 = 1 $, then $ M_1 ×f_1 M_2 $ becomes a  product of Finsler manifolds M1and M2.
 If $ f_2  = 1 $, then we have a  waperd product manifold. If $ f_i, i=1,2 $  is not constant, then we have a proper $DWP$-manifold.

Let $(M_1, F_1) $ and $(M_2, F_2)$ be two Finsler manifolds. Then the functions
\begin{equation*}
g_{ij}(x,y) = \dfrac{1}{2}\dfrac{\partial ^{2}F_{1}^{2}(x,y)}{\partial y^{i}\partial y^{j}},
\qquad g_{\alpha\beta}(u,v) = \dfrac{1}{2}\dfrac{\partial ^{2}F_{2}^{2}(u,v)}{\partial v^{\alpha}\partial v^{\beta}},
\end{equation*}
define a Finsler tensor field of type $(0, 2)$ on $TM^0_1$
 and $TM^0_2$, respectively. Now let $ M_1 \times _{f_1} M_2 $ be a  warped Finsler manifold and let $\textbf{x}\in M $ and $\textbf{y}  \in T_{\textbf{x}}M$, where $\textbf{x} = (x, u), \textbf{y}  = (y, v)$, $M = M_1 \times M_2$ and $T_{\textbf{x}}M = T_xM_1 \oplus T_uM_2$.
Then  we conclude that
\begin{equation*}
\textbf{g}_{ab}(x,u,y,v) = \left( \dfrac{1}{2}\dfrac{\partial ^{2}F^{2}(x,u,y,v)}{\partial \textbf{y}^{a}\partial \textbf{y}^{b}}\right) =\left(\begin{array}{cc}
             g_{ij}& 0 \\
              0& f_1^{2}g_{\alpha\beta}
            \end{array}\right)
\end{equation*}
where
$ \textbf{y}^{a}=(y^{i},v^{\alpha}),\; \textbf{y}^{b}=(y^{j},v^{\beta})$
and
$\textbf{g}_{ij}=g_{ij}, \;\textbf{g}_{ab}= f_1^{2}g_{\alpha\beta},\; \textbf{g}_{i\beta}=\textbf{g}_{\alpha j}=0 $
and
$$ i,j,... \in \{ 1,2,...,n_1\}  \alpha ,\beta,...\in\lbrace 1,2,...,n_2\rbrace,   a,b,...\in\lbrace 1,2,...n_1,n_1+1,...,n_1+n_2\rbrace ,$$
where
\begin{align*}
\dim (M_1)=n_1,\quad  \dim( M_2)=n_2,\quad \dim (M_1\times M_2 )= n_1+n_2.
\end{align*}
So the spray coeffcients of warped product are given by
\begin{align*}
\textbf{G}^{i}(x,u,y,v)&=G^{i}(x,y)-\dfrac{1}{4}g^{ih} \dfrac{\partial f_1^2}{\partial x^h}F^2_2,\\
\textbf{G}^{\alpha}(x,u,y,v)&=G^{\alpha}(u,v)+\dfrac{1}{4f_1^2}g^{\alpha \lambda} \dfrac{\partial f^2_1}{\partial x^l}\dfrac{\partial F_2^2}{\partial v^\lambda}y^l.
\end{align*}
The Berwald curvature of $ (M_1 \times _{f}M_2)$ is as follows:
\begin{align*}
&\textbf{B}^k_{ijl}=B^k_{ijl}-\dfrac{1}{4}\dfrac{\partial ^3g^{kh}}{\partial y^i \partial y^j\partial y^l}\dfrac{\partial f_1^2}{\partial x^h}F_2^2,& &\textbf{B}^\gamma_{\alpha\beta\lambda}=B^\gamma_{\alpha\beta\lambda},\\
&\textbf{B}^k_{i\beta l}=-\dfrac{1}{4}\dfrac{\partial ^2g^{kh}}{\partial y^l \partial y^i}\dfrac{\partial f_1^2}{\partial x^h}\dfrac{\partial F_2^2}{\partial v^\beta},&&\textbf{B}^\gamma_{i\beta \lambda}=0\\
&\textbf{B}^k_{\alpha\beta l}=-\dfrac{\partial f_1^2}{\partial x^h}\dfrac{\partial g^{kh}}{\partial y^l}g_{\alpha\beta},&&\textbf{B}^\gamma_{i j \lambda}=0,\\
&\textbf{B}^k_{\alpha\beta \lambda}=-\dfrac{\partial f_1^2}{\partial x^h}g^{kh}C_{\alpha\beta\lambda},&&\textbf{B}^\gamma_{ijk}=0.
\end{align*}
\begin{thm}[\cite{PeyghanTayebi:OndoublywarpedproductFinslerManifold}]
 Let ${M_1}_{f_2}\times _{f_1} M_2$ be  a  $DWP $-Finsler manifold and $f_1$ is constant on $M_1$ ($f_2$ is constant on $M_2$). Then ${M_1}_{f_2}\times _{f_1} M_2$ is Berwaldian if and only if $M_1$ is Riemannain, $M_2$ is Berwaldian and
$
C^{ij}_k\dfrac{\partial f_1}{\partial x^i}=0.
$
($M_2$ is Riemannian, $M_1$ is Berwaldian and $C^{ij}_k\dfrac{\partial f_2}{\partial x^i}=0$).
\end{thm}
\begin{cor}[\cite{PeyghanTayebi:OndoublywarpedproductFinslerManifold}]
Let $ (M_1\times _{f_1}M_2,F) $ be a proper $ WP $-Finsler manifold. Then $ (M_1\times _{f_1}M_2,F) $ is Berwaldian if and only if $ M_2 $ is Riemannian, $ M_1 $ is Berwaldian and
$$C^{ij}_k\dfrac{\partial f_1}{\partial x^i}=-2 \dfrac{\partial g^{ij}}{\partial y^k}\dfrac{\partial f_1}{\partial x^i}=0.$$
\end{cor}
\begin{thm}[\cite{peyghan.Tayebi:DoublywarpedproductFinslemanifolds}]
Let $ ({M_1} _{f_2} \times _{f_1}M_2,F) $ be a proper $ DWP $-Finsler manifold. Then $ ({M_1}_{f_2}\times _{f_1}M_2,F) $ is weakly Berwald  if and only if $ M_2 $ and $ M_1 $ are weakly Berwalds and
$$C^{ij}_k\dfrac{\partial f_1}{\partial x^i}=C^{\gamma \nu}_\gamma\dfrac{\partial f_2}{\partial u^\nu}=0.$$
\end{thm}
\begin{cor}[\cite{peyghan.Tayebi:DoublywarpedproductFinslemanifolds}]
Let $ (M_1\times _{f_1}M_2,F) $ be a proper $ WP $-Finsler manifold. Then $ (M_1\times _{f_1}M_2,F) $ is Douglas if and only if $ M_2 $ is Riemannian, $ M_1 $ is Berwaldian and
$$C^{ij}_k\dfrac{\partial f_1}{\partial x^i}=0.$$
\end{cor}
\begin{thm}
A DWP-Finsler manifold $ (M_1\times _{f_1}M_2,F) $ with isotropic mean Berwald curvature is a weakly Berwald manifold provided that $$C^{ij}_k\dfrac{\partial f_1}{\partial x^i}=0 \mbox{ or } C^{\gamma\nu}_\gamma\dfrac{\partial f_2}{\partial x^\nu}=0.$$
\end{thm}

\section{Main result}
\begin{thm}
If $ (M,F) $ is a 2-dimensional Finsler manifold and $ f $ is  nonconstant smooth function on $M$ satisfying
 \begin{align}
\dfrac{\partial g^{ij}}{\partial y^k}\dfrac{\partial f}{\partial x^i}=0,\label{pde}
\end{align}
then $ M $ is a Riemannian  manifold.
\end{thm}
\begin{proof}
Let $ f $ be a nonconstant  smooth function on $M$ which satisfies  \eqref{pde}. Then
we have
\begin{align}\label{m1}
g^{ij}\dfrac{\partial f}{\partial x^j}=c^i(x),
\end{align}
where $c^i(x)$ is a smooth function on $M$.
The above Equ. \eqref{m1} implies that
%\begin{align}
%&g^{ij}g_{ik}\dfrac{\partial f_1}{\partial x^j}=c^i(x)g_{ik}\nonumber\\
$\dfrac{\partial f}{\partial x^k}=c^i(x)g_{ik},$
%\end{align}
and
\begin{align}\label{m12}
&\dfrac{\partial f}{\partial x^j}=\frac{1}{2}c^i(x)\dfrac{\partial ^2F^2}{\partial y^i \partial y^j}.
\end{align}
By integrating  \eqref{m12} with respect to $ y^j $ we obtain
\begin{align}\label{m2}
 2\dfrac{\partial f}{\partial x^j}y^j=c^i(x)\dfrac{\partial F^2}{\partial y^i }.
\end{align}
By Choosing   $ u=F^2, A_j=\dfrac{\partial f}{\partial x^j}$ in the equation \eqref{m2}, we have the following  PDE equation
\begin{align}\label{m4}
c^1\dfrac{\partial u}{\partial y^1}+c^2\dfrac{\partial u}{\partial y^2}=2A_1y^1+2A_2y^2.
\end{align}

%\begin{align}\label{m5}
%\dfrac{d y^1}{c^1}=\dfrac{d y^2}{c^2}=\dfrac{d u}{2B_1y^1+2B_2y^2}
%\end{align}
The general solution of PDE equation \eqref{m4} are given by
\begin{align}
F^2=u=\dfrac{A_1}{c^1}{(y^1)}^2+\dfrac{A_2}{c^2}{(y^2)}^2+\varphi(c^2y^1-c^1y^2)
\end{align}
where  $\varphi$ is an arbitrary smooth one variable function \cite{sneddon:PDE2006}.
Since   $ u$ is homogeneous of degree $ 2 $ so $ \varphi $ is homogeneous of degree $ 2 $, too.
That  $\varphi$ is one variable  implies that $ \varphi (t)=at^2 $, where $a$ is a real constant,  so we have
 \begin{align}
 F^2=\dfrac{A_1}{c^1}{(y^1)}^2+\dfrac{A_2}{c^2}{(y^2)}^2+a(c^2y^1-c^1y^2)^2.
 \end{align}
The Cartan torsion $$C_{ijk}=\frac{1}{2}\frac{\partial g_{ij}}{\partial y^k}=\frac{1}{4}\frac{\partial^3 F^2}{\partial y^i\partial y^j\partial y^k}=0$$ thus $ M $ is Riemannian.
\end{proof}
\begin{cor}
Let $(M_1,F_1)$  and $ (M_2,F_2)$ be Finsler manifolds with $\dim M_1=2,\; \dim M_2=n_2 $ and $f_i:M_i\to \mathbb{R}, i=1,2$ are  positive smooth functions.
\begin{enumerate}
  \item
A proper $(2+n_2)$-dimensional $ WP $-Finsler manifold $M_1\times_{f_1} M_2$ is a Berwald  manifold, if and only if
it is a Riemannian manifold.
\item
A proper $(2+2)$-dimensional $ DWP $-Finsler manifold ${M_1}_{f_2}\times_{f_1} M_2$ is a weakly Berwald manifold, if and only if
it is Riemannian manifold ($\dim M_2=2$).
\item
A proper $(2+n_2)$-dimensional $ WP $-Finsler manifold $M_1\times_{f_1} M_2$ is a Douglas  manifold, if and only if
it is a Riemannian manifold.
%\item
%A proper $(2+n_2)$-dimensional $ WP $-Finsler manifold $M_1\times_f M_2$ is Douglas  manifold, if and only if
%it is Riemannian manifold, where $, \dim M_1=2, \dim M_2=n_2 $.
\end{enumerate}
\end{cor}
The authors have not succeeded in finding a counterexample to the following problem. They conjecture that it might be true but, unfortunately, they have been unable to provide a proof for it.\\
\subsubsection*{Problem.}
Is any  $n$-dimensional($n>2$) Finsler manifold  $ (M,F) $ with  function $f$ which is nonconstant  and smooth on $M$  satisfying
$\dfrac{\partial g^{ij}}{\partial y^k}\dfrac{\partial f}{\partial x^i}=0,$
a Riemannian  manifold?

%\begin{proof}
%It is obvious that every Riemannian manifol is a Berwald manifold.\\
%Conversely, suppose warped product is Berwaldian.  theorem [] yields $ (M_2,F_2) $
%is Riemannian, $ (M_1,F_1) $ is Berwaldian and $ \dfrac{\partial g^{ij}}{\partial y^k}\dfrac{\partial f_1}{\partial x^i}=0 $,
%But theorem $[]$ implies that $ (M_1,F_1) $ is Riemannian too. so the warped product
%$ (M_1 \times _{f}M_2) $ is Riemannian.
% \end{proof}
%it is an open problem for n-dimensional proper $ WP $-Finsler manifold
%\end{proof}
%\begin{align}
%\left\lbrace\begin{array}{c}\label{m3}
%\dfrac{d y^1}{c^1}=\dfrac{dy^2}{c^2} \\
%\vdots\\
%\dfrac{d y^1}{c^1}=\dfrac{dy^n}{c^n}
%\end{array}\right.
%\end{align}
%Integraling \ref{m3} with respect to ? implies that
%\begin{align*}
%\left\lbrace \begin{array}{c}
%c^2y^1-c^1y^1=\alpha _2 \\
%.\\
%. \\
%. \\
%c^ny^1-c^1y^n=\alpha _n
%\end{array}\right.
%\end{align*}
%This paper is arranged as follows: In section 2, we give some basic concepts of
%Finsler manifolds and Berwald metrics. In sections 3  we study Berwald warped product Finsler
%metrics or briefly WP-Finsler metrics with vanishing Berwald curvature.
%%%%%%%%%%%%%%%%%%%%%%%%%%
%\section*{Acknowledgements }
%The authors is very grateful Professor Akbar Tayebi for some
%valuable suggestions to improve the presentation of this paper.
%\nocite{*}
%\bibliographystyle{siam}
%\bibliography{D:/bibtex/convolution,D:/bibtex/faghfouri}

\def\polhk#1{\setbox0=\hbox{#1}{\ooalign{\hidewidth
  \lower1.5ex\hbox{`}\hidewidth\crcr\unhbox0}}}

\end{document}